\documentclass[11pt,leqno]{article}

\usepackage{amsthm,amsfonts,amssymb,amsmath,oldgerm}
\usepackage{fullpage}
\usepackage{graphicx}
\usepackage{mathrsfs}
\usepackage{xcolor}

\numberwithin{equation}{section}

\renewcommand\d{\partial}

\def\eps {\varepsilon}

\renewcommand{\L}{\mathcal{L}}

\newcommand{\RM}{{\mathbb{R}}}
\newcommand{\CM}{{\mathbb{C}}}
\newcommand{\NM}{{\mathbb{N}}}
\newcommand{\ZM}{{\mathbb{Z}}}
\newtheorem{theorem}{Theorem}[section]
\newtheorem{proposition}[theorem]{Proposition}

\newtheorem{lemma}[theorem]{Lemma}
\newtheorem{remark}[theorem]{Remark}
\theoremstyle{definition}
\newtheorem{definition}[theorem]{Definition}
\newtheorem{assumption}[theorem]{Assumption}

\allowdisplaybreaks[3]

\title{Subharmonic Dynamics of Wave Trains in the Korteweg-de Vries / Kuramoto-Sivashinsky Equation}

\author{Mathew~A.~Johnson\thanks{Department of Mathematics, University of Kansas, 1460 Jayhawk Boulevard, 
Lawrence, KS 66045, USA; matjohn@ku.edu}\quad\&\quad Wesley R. Perkins\thanks{Department of Mathematics, Lehigh University, 17 Memorial Drive East, 
Bethlehem, PA 18015, USA; wesley.perkins@lehigh.edu} }

\date{\today}

\begin{document}

\maketitle

\begin{abstract}
We study the stability and nonlinear local dynamics of spectrally stable periodic wave trains of the Korteweg-de Vries / Kuramoto-Sivashinsky equation
when subjected to classes of periodic perturbations.  It is known that for each $N\in\NM$, such a $T$-periodic wave train is asymptotically stable to $NT$-periodic, i.e., subharmonic, perturbations, in the sense
that initially nearby data will converge asymptotically to a small Galilean boost of the underlying wave, with exponential rates of decay.
However, both the allowable size of initial perturbations and the exponential rates of decay depend on $N$ and, in fact, tend to zero as $N\to\infty$, leading
to a lack of uniformity in such subharmonic stability results.  Our goal here is to build upon a recent methodology introduced by the authors in the reaction-diffusion setting
and achieve a subharmonic stability result which is uniform in $N$.   This work is motivated by the dynamics of such wave trains when subjected to perturbations
which are localized (i.e., integrable on the line).
\end{abstract}

\section{Introduction}\label{S:intro}

In this work, we consider the local dynamics of wave trains, i.e., periodic traveling wave solutions, of the Korteweg-de Vries / Kuramoto-Sivashinsky (KdV/KS) equation
\begin{equation}\label{e:ks}
u_t+\eps u_{xxx}+\delta\left(u_{xx}+u_{xxxx}\right)+uu_x=0
\end{equation}
where $x,t\in\RM$.  Here, $\eps\geq 0$ and $\delta>0$ are modeling parameters which may, without loss of generality, be chosen such that $\eps^2+\delta^2=1$: see
Remark \ref{r:scaling} below.
The equation \eqref{e:ks} is known to be a canonical model for pattern formation that has been used to describe many applications including 
plasma instabilities, turbulence in reaction diffusion equations, flame front propagation,
and nonlinear wave dynamics in fluid mechanics \cite{S1,S2,SM,K,KT}.  In the case $\eps=0$ and $\delta=1$, equation \eqref{e:ks} becomes the classical Kuramoto-Sivashinsky equation, which
is known to be a generic equation for chaotic dynamics, and there is a very large literature on these solutions, their bifurcations and period doubling cascades, and their stability:
see, for example, \cite{CD} and reference therein.  
In this $\eps=0$ case, \eqref{e:ks} is also known to model thin film dynamics down a completely vertical wall \cite{FST}.  When the angle of the wall is decreased from vertical, however,
additional dispersive effects are present \cite{CD} and modeled by $\eps>0$, and in the ``flat" limit where the inclined wall becomes horizontal one recovers the completely
integrable Korteweg-de Vries
equation, corresponding here to $\eps=1$ and $\delta=0$.  Thus, in the context of inclined thin film flow one can consider the general model \eqref{e:ks} as interpolating between the 
vertical wall $(\eps,\delta)=(0,1)$ and the ``flat" limit $(\eps,\delta)\to(1^-,0^+)$.  For more information on the connection to thin film dynamics, as well as its derivation in this
context from the viscous shallow water equations or the full Navier Stokes equations, see \cite{Wi,YY}.

\begin{remark}\label{r:scaling}
In the literature, one may encounter more general looking systems of the form  
\[
u_t+\eps u_{xxx}+ \delta u_{xx} + \gamma u_{xxxx} + \Lambda uu_x=0
\]
where here $\eps\in\RM$ and $\delta,\gamma,\Lambda>0$ are arbitrary constants.  However, we note that, through a rescaling argument, such systems can always be
put in the form \eqref{e:ks}, i.e., one can always take $\Lambda=1$ and $\delta=\gamma$ with $\eps^2+\delta^2=1$.  See \cite[Section 2]{BJNRZ13}.
Note the particular scaling here is sometimes referred to as the ``thin film" scaling: see \cite{CD} for example.
\end{remark}

In this work, we are interested in understanding the stability and long-time dynamics of periodic traveling wave solutions of \eqref{e:ks}
to specific classes of perturbations.  To begin our discussion, we briefly discuss the existence theory for 
periodic solutions of \eqref{e:ks}.  First, note that traveling wave solutions of \eqref{e:ks} correspond to solutions of the form $u(x,t) = \bar{u}(x-ct)$, where
$\bar{u}$ necessarily satisfies the profile ODE
\begin{equation}\label{e:profile}
-c\bar{u}'+\eps\bar{u}'''+\delta\left(\bar{u}''+\bar{u}''''\right)+\bar{u}\bar{u}'=0.
\end{equation}
The existence and local structure of periodic solutions of \eqref{e:profile} has been studied extensively by several other authors.  For general
values of $(\eps,\delta)$, an elementary Hopf bifurcation analysis \cite{BJNRZ13} shows the existence of 
a $3$-parameter family of asymptotically small amplitude periodic traveling wave solutions of \eqref{e:profile} which, up to translation, can be parametrized by the wave
speed $c$ and the period $T$.  In the ``classical KS'' limit $|\eps|\ll 1$, one can likewise use normal form analysis to establish a similar existence result \cite{CD}, while
the full bifurcation picture for the KS equation ($\eps=0$) is known to be extremely complicated: see, for example, \cite{KE}, where the authors prove the existence of a Shi'lnikov bifurcation
which leads to cascades of period doubling, period multiplying $k$-bifurcations and oscillatory homoclinic orbits as the period is increased, as well as the numerical bifurcation
study in \cite{BKJ}.  For a summary and more details, see \cite{BJNRZ13,JNRZ15}.

Using the above existence studies as motivation, and closely following the work in \cite{BJNRZ13}, we make the following assumption regarding the existence of periodic solutions of \eqref{e:profile},
as well as the structure of the local manifold of periodic solutions.

\begin{assumption}\label{a:exist}
Suppose $\bar{u}(\cdot)$ is a $T$-periodic solution of \eqref{e:profile} with speed $c$.  Then the set of periodic solutions near $\bar{u}$ forms
a $3$-dimensional manifold
\[
\left\{(x,t)\mapsto u(x-\gamma-(c+\delta c)t;X):\gamma\in\RM,~(\delta c,X)\in\Omega\right\}
\]
where $\Omega\subset\RM^{2}$ is some open set containing the point $(\delta c, X) = (0,T)$ and each $u(\cdot;X)$ is an $X$-periodic solution to \eqref{e:profile} with 
wave speed $c+\delta c$.
\end{assumption}

\begin{remark}
It is natural to assume \eqref{e:profile} admits a $3$-dimensional manifold of periodic solutions.  Indeed, note that integrating the profile equation \eqref{e:profile}
once yields
\[
-c\bar{u}+\eps\bar{u}''+\delta\left(\bar{u}'+\bar{u}'''\right)+\frac{1}{2}\bar{u}^2=q
\]
for some constant of integration $q\in\RM$.    Periodic solutions of \eqref{e:profile} thus correspond to values
\[
(T,c,q,\bar{u}(0),\bar{u}'(0),\bar{u}''(0))\in\RM^6,
\]
where $T$, $c$, and $q$ denote the period, wave speed and constant of integration, subjected to the periodicity condition
\[
\left(\bar{u}(T),\bar{u}'(T),\bar{u}''(T)\right)=\left(\bar{u}(0),\bar{u}'(0),\bar{u}''(0)\right).
\]
leading, generically, to a $3$-parameter family of periodic solutions parametrized, up to translation, by the period $T$, the wave speed $c$.
\end{remark}

%
%
%
%

\

As stated above, our main goal is to study the stability and dynamics of periodic traveling wave solutions of \eqref{e:ks} to specific classes of perturbations.
Previously, there has been much work regarding both the spectral and nonlinear stability of such
periodic solutions when subjected to localized perturbations, i.e., perturbations that are integrable on the line \cite{BJNRZ13,BJNRZ12,JNRZ15}, as well
as their dynamics under slow-modulations \cite{NR13}.  In these works, it is found that for each admissible pair $(\eps,\delta)$ in \eqref{e:ks}
there exist periodic traveling wave solutions which are spectrally and nonlinearly stable to localized perturbations.
Here, however, we study the stability and dynamics of $T$-periodic traveling wave solutions of \eqref{e:ks} 
when subject to so-called \emph{subharmonic perturbations}, i.e., $NT$-periodic perturbations for some $N\in\NM$.
Before we continue, we note that an extremely important feature of \eqref{e:ks}, which we utilize heavily in our forthcoming analysis, 
is the presence of a Galilean symmetry.  In particular, if $u(x,t)$ is a solution of \eqref{e:ks}, then so is the function
\begin{equation}\label{galilean}
u(x-ct,t)+c
\end{equation}
for any $c\in\RM$.
Thanks to this Galilean invariance, as well as the translational invariance of \eqref{e:ks},
 it follows that the stability of a particular wave only depends on one parameter: namely, the period $T$.  Thus, when
discussing the stability of periodic traveling wave solutions of the KdV-KS equation \eqref{e:ks}, we identify all waves of a particular period.  Furthermore, note
we can view \eqref{galilean} as coupling between the wave speed and the mass of the wave, defined via
\[
M:=\int_0^Tu(x,t)dx,
\]
which is readily seen to be a conserved quantity of \eqref{e:ks} due to the conservative structure.

Now, suppose that $\bar{u}(x)$ is a $T$-periodic solution of \eqref{e:profile} with wave speed $c$, and note the linearization of \eqref{e:ks} (in the appropriate co-moving frame) about $\bar{u}$ 
is given by the operator
\[
\L:=\partial_x\left(c-\bar{u}-\partial_x^2\right)-\delta\left(\partial_x^2+\partial_x^4\right). 
\]
Since we are considering subharmonic perturbations of $\bar{u}$, i.e., perturbations with period $NT$ for some $N\in\NM$, here we consider
$\L$ as a closed, densely defined linear operator acting on $L^2_{\rm per}(0,NT)$ with $T$-periodic coefficients.  To describe the spectrum
of $\L$ acting on $L^2_{\rm per}(0,NT)$, first observe that one always has that
\begin{equation}\label{ker1}
\L\bar{u}_x=0\quad{\rm and}\quad \L(1)=-\bar{u}_x
\end{equation}
so that $\lambda=0$ is an eigenvalue of $\L$ with algebraic multiplicity at least two and a Jordan chain of height at least one.  
With this in mind, and following the works \cite{GS97,GS96,JNRZ_Invent,JZ11},  we introduce the notion of spectral stability used throughout this work.

\begin{definition}\label{d:spec_stab}
A $T$-periodic traveling wave solution $\bar{u}\in H^1_{\rm loc}(\RM)$ of \eqref{e:profile} is said to be \emph{diffusively spectrally stable} provided the following
conditions hold:
\begin{itemize}
\item[(D1)] The spectrum of the linear operator $\L$ acting on $L^2(\RM)$ satisfies
\[
\sigma_{L^2(\RM)}\left(\L\right)\subset\left\{\lambda\in\RM:\Re(\lambda)<0\right\}\cap\{0\};
\]
\item[(D2)] There exists a $\theta>0$ such that for any $\xi\in[-\pi,\pi)$ the real part of the spectrum of the Bloch operator $\L_\xi:=e^{-i\xi x}\L e^{i\xi x}$
acting on $L^2_{\rm per}(0,1)$ satisfies
\[
\Re\left(\sigma_{L^2_{\rm per}(0,1)}\left(\L_\xi\right)\right)\leq-\theta\xi^2;
\]
\item[(D3)] $\lambda=0$ is a $T$-periodic eigenvalue of $\L_0[\phi]$ with algebraic multiplicity two and geometric multiplicity one. 
\end{itemize}
\end{definition}

\begin{figure}[t]
\begin{center}
\includegraphics[scale=1.3]{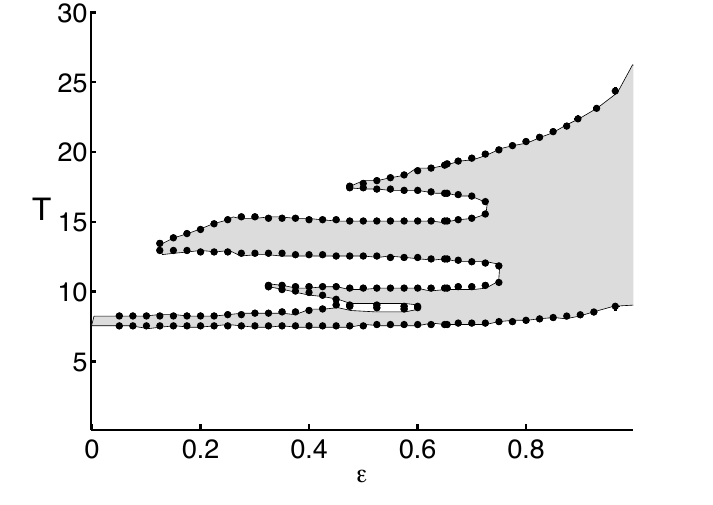}
\caption{Plot of the numerically obtained stability boundaries, computed in \cite[Section 2]{BJNRZ13},
in the period (denoted here as $X$) versus the parameter $\eps$ in \eqref{e:ks}.  Note that $\delta=\sqrt{1-\eps^2}$
is fixed by the choice of $\eps$ and that, by Galilean invariance, the stability of such a wave depends only on its period.  
In this figure, the shaded regions correspond to diffusively spectrally stable periodic traveling wave solutions of \eqref{e:ks}.
}\label{f:stab_bdry}
\end{center}
\end{figure}

We note that in \cite{BJNRZ13} it was shown that for every admissible pair of modeling parameters $(\eps,\delta)$ there exists a range of periods for which
all periodic traveling wave solutions of \eqref{e:ks} with those period are diffusively spectrally stable: see Figure \ref{f:stab_bdry} and also \cite{BAN,CDK}.
In the context of the KdV/KS equation \eqref{e:ks}, it is known that diffusively spectrally stable periodic traveling wave solutions (coupled with an additional
non-degeneracy hypothesis, see Assumption \ref{a:non-deg} below) are nonlinearly stable
to localized perturbations, in the sense that localized perturbations of such a periodic waves $\bar{u}$ converges to spatially localized phase modulations of $\bar{u}$.
Specifically, given such a periodic traveling wave $\bar{u}$ and initial data
\[
u(x,0) = \bar{u}(x)+v(x,0),\quad \|v(\cdot,0)||_{L^1(\RM)\cap H^5(\RM)}\ll 1
\]
then for large time the associated solution $u(x,t)$ satisfies
\[
u(x,t)\approx \bar{u}(x+\Psi(x,t)),~~~t\gg 1
\]
for some function $\Psi(\cdot,t)\in W^{5,\infty}(\RM)$ which behaves essentially like a finite sum of (small) error functions: see \cite{BJNRZ13}.
As we will see below, the implication of diffusive spectral stability will be that such a $T$-periodic traveling wave is necessarily spectrally stable
to perturbations in $L^2_{\rm per}(0,NT)$ for every $N\in\NM$.

Before continuing to state our main result, we introduce an additional non-degeneracy hypothesis.  As we will see in Lemma \ref{l:spec_prep} below, 
Assumption \ref{a:exist}, along with the diffusive spectral stability assumption, implies that for $|\xi|$ small the Bloch operators $\L_\xi$ have
two eigenvalues near the origin which expand as
\begin{equation}\label{e:crit_bif}
\lambda_j(\xi)=-ia_j\xi+o\left(\xi\right),~~a_j\in\RM,~~j=1,2
\end{equation}
for $|\xi|\ll 1$.
Following the previous work \cite{BJNRZ13,JNRZ_Invent}, we make the following additional  non-degeneracy hypothesis:

\begin{assumption}\label{a:non-deg}
The coefficients $a_j\in\RM$ in \eqref{e:crit_bif} are distinct.
\end{assumption}

By standard spectral perturbation theory, Assumption \ref{a:non-deg} ensures the analyticity of the functions $\lambda_j(\cdot)$ in \eqref{e:crit_bif}:
again, see Lemma \ref{l:spec_prep} below.
Furthermore, it is known that the coefficients $a_j$ are the characteristics of an associated Whitham averaged system, formally governing
slowly modulated periodic solutions of \eqref{e:ks}.  Consequently, Assumption \ref{a:non-deg} corresponds to strict hyperbolicity of the
Whitham averaged system.  See \cite{NR13} for more details.  Note also that all the diffusively spectrally stable periodic traveling wave solutions
in Figure \ref{f:stab_bdry} were seen numerically to satisfy Assumption \ref{a:non-deg}.

\

We now begin our discussion of our main results concerning the dynamics of $T$-periodic, diffusively spectrally stable traveling wave solutions
of \eqref{e:ks} when subjected to subharmonic perturbations.
 First, in Section \ref{S:fb} we will see that
the spectrum of $\L$ acting on $L^2_{\rm per}(0,NT)$ is given by the union of the (necessarily discrete) spectrum of the corresponding 
Bloch operators $\L_\xi$ acting in $L^2_{\rm per}(0,T)$ for the discrete (finite) subset of $\xi\in[-\pi/T,\pi/T)$ such that $e^{i\xi NT}=1$.  
Thus, such a traveling wave $\bar{u}$ is necessarily spectrally stable to all subharmonic perturbations. 
In particular, for each $N\in\NM$  there exists a constant $\delta_N>0$ such that 
\[
\Re\left(\sigma_{L^2_N}\left(\L\right)\setminus\{0\}\right)\leq-\delta_N.
\]
Since $\L$ is clearly sectorial, it is easy to show that for $N\in\NM$ fixed and $\delta\in(0,\delta_N)$
there exists a constant $C_\delta>0$ such that
\begin{equation}\label{e:exp_lin_est}
\left\|e^{\L t}\left(1-\mathcal{P}_N\right)f\right\|_{L^2_N}\leq Ce^{-\delta t}\|f\|_{L^2_N}
\end{equation}
for all $f\in L^2_N$, where here $\mathcal{P}_N$ denotes the projection of $L^2_N$ onto the $NT$-periodic generalized kernel 
of $\L$.  Equipped with this linear estimate and exploiting the Galilean invariance \eqref{galilean}, 
the following nonlinear stability result was established in \cite{SS11,SS11_2} for the case $N=1$ and is easily extended to general $N\in\NM$.

\begin{proposition}[Nonlinear Subharmonic Stability \cite{SS11,SS11_2}]\label{p:sub_fixedN}
Let $\bar{u}\in H^1_{\rm loc}$ be a $T$-periodic traveling wave solution of \eqref{e:ks} with wave speed $c$.  Assume that $\bar{u}$ is 
diffusively spectrally stable and additionally satisfies Assumption \ref{a:non-deg}.  Fix $N\in\NM$ and take $\delta_N>0$ such that
\[
\max\left(\Re\left(\sigma_{L^2_{\rm per}(0,NT)}\left(\L\right)\right)\setminus\{0\}\right)=-\delta_N
\]
holds.  Then for each $N\in\NM$ and for every $\delta\in(0,\delta_N)$, there exists an $\eps=\eps_\delta>0$ and a constant $C=C_\delta>0$
such that whenever $u_0\in H^1_{\rm per}(0,NT)$ and $\|u_0-\bar{u}\|_{H^1(0,NT)}<\eps$, then the solution $u$ of \eqref{e:ks} with initial
data $u(0)=u_0$ exists globally in time and satisfies
\[
\left\|u(\cdot+\gamma_\infty,t) - \bar{u}\left(\cdot-(\delta M)t\right)-\delta M\right\|_{H^1(0,NT)}\leq C e^{-\delta t}\left\|u_0-\bar{u}\right\|_{H^1(0,NT)}
\]
for all $t>0$, where here $\gamma_\infty=\gamma_\infty(N)$ is some constant and
\[
\delta M:=\frac{1}{NT}\int_0^{NT}\left(u_0(z)-\bar{u}(z)\right)dz.
\]
\end{proposition}

The above result establishes the nonlinear asymptotic stability of the Galilean family associated to $\bar{u}$, showing that nearby $NT$-periodic solutions will, up to a spatial translation,
asymptotically converge to a member of the Galilean family of $\bar{u}$.  However, Proposition \ref{p:sub_fixedN} lacks uniformity in $N$ in two important ways.
Specifically, both the exponential rate of decay $\delta$ and the allowable size of initial perturbations $\eps=\eps_\delta$ are controlled by the size of the
spectral gap $\delta_N>0$.  Since it is known that $\delta_N\to 0$ as $N\to\infty$, it follows that both $\delta$ and $\eps$ in Proposition \ref{p:sub_fixedN} must
necessarily tend to zero as $N\to\infty$.  Of course, from a practical level it would be preferable to develop a theory which, for a given background wave $\bar{u}$,
yielded a fixed size for the theoretically prescribed domain of attraction as well as a uniform (in $N$) rate of decay of initial perturbations.
The fact that this is possible is precisely our main result.

\begin{theorem}[Uniform Subharmonic Asymptotic Stability]\label{T:main}
Let $\bar{u}\in H^1_{\rm loc}$ be a $T$-periodic traveling wave solution of \eqref{e:ks} with wave speed $c$.  Assume that $\bar{u}$ is 
diffusively spectrally stable and, additionally, satisfies Assumption \ref{a:non-deg}. There exists an $\eps>0$ and a constant $C>0$ 
such that, for every $N\in\NM$, whenever $u_0\in L^1_{\rm per}(0,NT)\cap H^5_{\rm per}(0,NT)$ and
\[
E_0:=\left\|u_0-\bar{u}\right\|_{L^1_{\rm per}(0,NT)\cap H^5_{\rm per}(0,NT)}<\eps,
\]
there exists a function $\widetilde{\psi}(x,t)$ satisfying $\widetilde{\psi}(\cdot,0)\equiv 0$ such that the solution $u$ of \eqref{e:ks} with
initial data $u(0)=u_0$ exists globally in time and satisfies
\begin{equation}\label{e:stab_estimate} 
\left\|u\left(\cdot-\widetilde{\psi}(\cdot,t),t\right)-\bar{u}\left(\cdot-(\delta M)t\right)-\delta M\right\|_{H^5_{\rm per}(0,NT)},
	~~\left\|\nabla_{x,t}\widetilde{\psi}(\cdot,t)\right\|_{H^5_{\rm per}(0,NT)}\leq CE_0(1+t)^{-1/4}
\end{equation}
for all $t\geq 0$.  
\end{theorem}

The main difficulty in establishing the proof of Theorem \ref{T:main} appears at the linear level.  Specifically, one must develop a strategy
to handle the accumulation of $NT$-periodic eigenvalues near $\lambda=0$ for $N\gg 1$ in a uniform way.  
In the proof of Proposition \ref{p:sub_fixedN}, this accumulation is handled by enclosing the origin in the spectral plane
in an small ball $B(0,r_N)$ where the $N$-dependent radius is chosen so that
\begin{equation}\label{e:exp_spec_decomp}
\sigma_{L^2_{\rm per}(0,NT)}\left(\L\right)\cap B(0,r_N)=\{0\}
\end{equation}
and defining the associate Riesz spectral projection
\begin{equation}\label{e:riesz}
\mathcal{P}_N=\frac{1}{2\pi i}\int_{\partial B(0,r_N)}\frac{dz}{z-\L}
\end{equation}
onto the $NT$-periodic generalized kernel of $\L$.  One then decomposes the linear solution operator (semigroup) as
\[
e^{\L t} = e^{\L t}\mathcal{P}_N + e^{\L t}\left(1-\mathcal{P}_N\right)
\]
and uses the exponential bound \eqref{e:exp_lin_est} to establish the result.  The lack of uniformity in Proposition \ref{p:sub_fixedN}, however, stems from the fact
that the radius $r_N$ of the ball used in the Riesz projection \eqref{e:riesz} must necessarily tend to zero as $N\to\infty$ in order to maintain the
spectral decomposition \eqref{e:exp_spec_decomp}.  

To establish the uniformity in Theorem \ref{T:main}, however, we cannot allow the radius of the ball about the origin to shrink.  
Instead, we define a ball $B(0,r)$ \emph{with an $N$-independent radius} and define the associated
Riesz spectral projection associated to the total generalized eigenspace  corresponding to eigenvalues interior to the ball.  
Naturally, the dimension of this total eigenspace is tending
to infinity as $N\to\infty$, and we must work to establish uniform in $N$ decay estimates associated with the induced decomposition of the semigroup.
This methodology
was first introduced, at the linear level, by the authors and collaborators in \cite{HJP} and was further extended, in the context of reaction-diffusion
systems, to the full nonlinear level in \cite{JP21}, and is closely modeled off of the known stability theory for localized perturbations \cite{JNRZ_Invent}.
In the present work, the additional complication (compared to these previous subharmonic works) is the presence of a non-trivial Jordan block \eqref{ker1}
which, in turn, yields slower uniform decay rates of the associated linear semigroups as compared to those rates in the reaction-diffusion context.

Finally, we make a few remarks concerning the connection between Proposition \ref{p:sub_fixedN} and Theorem \ref{T:main}.  
The analysis in Section \ref{S:nonlinear} shows that the modulation function $\widetilde{\psi}$ in Theorem \ref{T:main} can be decomposed as
\[
\widetilde{\psi}(x,t) = \frac{1}{N}\gamma(t)+\psi(x,t)
\]
where here $\psi(\cdot,t)\in W^{5,\infty}_N$ satisfies
\begin{equation}\label{e:stab_bd}
\left\|\psi(\cdot,t)\right\|_{L^\infty_N}\leq CE_0
\end{equation}
for some constant $C>0$ independent of $N$.  Further, while our methods fail to directly yield convergence\footnote{This is due to the presence of a non-trivial Jordan block
associated with the co-periodic Bloch opertaor $\L_0$, leading to slower linear decay rates in Section \ref{S:lin} when compared to, for example, the reaction-diffusion context
where no such Jordan block is present (see \cite{JP21} for details).} of $\gamma(t)$ as $t\to\infty$, one immediate observation is that
for each fixed $N\in\NM$ and $NT$-periodic solution $u(x,t)$ of \eqref{e:ks} with initial data sufficiently close (in $L^1_N\cap H^5_N$) to $\bar{u}$, then using
the notation from both Proposition \ref{p:sub_fixedN} and Theorem \ref{T:main} we have
\[
\left\|u(\cdot+\gamma_\infty,t)-u(x-\widetilde{\psi}(\cdot,t),t)\right\|_{L^2_N}\leq C_N(1+t)^{-1/4}
\]
for some ($N$-dependent) constant $C_N>0$.  This leads one to expect that, under relatively generic circumstances,  for each fixed $N$ one should have
\[
\frac{1}{N}\gamma(t)+\psi(x,t)\to \gamma_\infty
\]
as $t\to\infty$.  Understanding this convergence rigorously is an interesting remaining problem from our analysis.

\

The outline of this paper is as follows.  In Section \ref{S:prelim} we briefly review the application of Floquet-Bloch theory
to subharmonic perturbations and collect several properties of the Bloch operators and their associated semigroups.  We further
establish basic ``high-frequency" decay properties of the Bloch semigroups that arise as consequences of the diffusive spectral stability assumption.
In Section \ref{S:lin} we provide a delicate decomposition of the semigroup $e^{\L t}$ acting on $L^2_{\rm per}(0,NT)$ which will yield
polynomial  decay rates on the semigroup which are \emph{uniform in $N$}: see Proposition \ref{p:lin_est}.  This linear decomposition
then motivates in Section \ref{S:nonlinear} a nonlinear decomposition of a small $L^2_{\rm per}(0,NT)$ neighborhood of the underlying diffusively stable $T$-periodic
background wave $\bar{u}$.  Equipped with this nonlinear decomposition, along with the linear estimates from Section \ref{S:lin}, we apply a nonlinear iteration scheme
to corresponding system of perturbation equations, yielding a proof of our main result Theorem \ref{T:main} above.

%

\

\noindent
{\bf Acknowledgments:}  The work of MAJ was partially funded by the NSF under grant DMS-2108749, as well as the Simons Foundation Collaboration grant number 714021.

\section{Preliminaries}\label{S:prelim}

Here, we review several preliminary analytical results.  We begin with a brief review of the application of Floquet-Bloch theory to subharmonic perturbations, and then
use this to establish some elementary semigroup estimates for the associated Bloch operators.  For notational convenience, for each $N\in\NM$, $p\geq 1$, and fixed period $T>0$
we introduce the notation
\[
L^p_N:=L^p_{\rm per}(0,NT)
\]
and similarly with all associated Sobolev spaces.

\subsection{Floquet-Bloch Theory for Subharmonic Perturbations}\label{S:fb}

We begin by reviewing the results of Floquet-Bloch theory when applied to subharmonic perturbations.  For more details, see the recent works \cite{HJP,JP21}.

Suppose that $\bar{u}$ is a $T$-periodic traveling wave solution of \eqref{e:ks} with wave speed $c\in\RM$, and consider the associated linearized operator
$\L$.  For each fixed $N\in\NM$, define\footnote{Observe that $\Omega_N$ is always a finite set with $|\Omega_N|=N$ and with the distance between any two closest
elements being $2\pi/NT$.}
\[
\Omega_N:=\left\{\xi\in[-\pi/T,\pi/T):e^{i\xi NT}=1\right\},
\]
and note that, since the coefficients of $\L$ are $T$-periodic, basic results in Floquet-theory implies that any $NT$-periodic solution of the ordinary differential equation
\[
\L v=\lambda v
\]
must be of the form
\[
v(x) = e^{i\xi xT} w(x)
\]
for some $\xi\in\Omega_N$ and $w\in L^2_{\rm per}(0,T)$.  More specifically, one can show that $\lambda\in\CM$ is an $NT$-periodic 
eigenvalue of $\L$ if and only if there exists a $\xi\in\Omega_N$ such that the problem
\[
\lambda w= e^{-i\xi x}\L e^{i\xi x}w =:\L_\xi w,
\]
admits a non-trivial solution in $L^2_{\rm per}(0,T)$, where here the operators $\L_\xi$ are known as the Bloch operators associated to $\L$ 
and the parameter $\xi$ is referred to as the Bloch frequency: see also Definition \ref{d:spec_stab}.  In fact, we have the spectral decomposition
\[
\sigma_{L^2_N}\left(\L\right) = \bigcup_{\xi\in\Omega_N}\sigma_{L^2_{\rm per}(0,T)}\left(\L_\xi\right),
\]
which characterizes the $N$-periodic spectrum of $\L$ in terms of the union of a $T$-periodic eigenvalues for the $1$-parameter family of Bloch operators $\{\L_\xi\}_{\xi\in\Omega_N}$.

From above, it is natural when studying such $NT$-periodic spectral problems to desire to decompose arbitrary functions as superpositions of functions
of the form $e^{i\xi \cdot}w(\cdot)$ with $w\in L^2_{\rm per}(0,T)$.  To this end, given a function $g\in L^2_N$ we define the $T$-periodic Bloch transform
of $g$ as
\[
\mathcal{B}_T(g)(\xi,x):=\sum_{\ell\in\ZM}e^{2\pi i\ell x/T}\hat{g}\left(\xi+2\pi\ell/T\right),~~\xi\in\Omega_N,~~x\in\RM
\]
where now $\hat{g}$ denotes the Fourier transform of $g$ on the torus given by
\begin{equation}\label{e:Fourier}
\hat{g}(z):=\int_{-NT/2}^{NT/2} e^{-izy}g(y)dy.
\end{equation}
In particular, observe that for any $\xi\in\Omega_N$ and $g\in L^2_N$, the function $\mathcal{B}_T(g)(\xi,\cdot)$ is $T$-periodic and, furthermore, the function
$g$ can be recovered via the inverse Bloch representation formula
\[
g(x)=\frac{1}{NT}\sum_{\xi\in\Omega_N} e^{i\xi x}\mathcal{B}_T(g)(\xi,x).
\]
One can check that the $T$-periodic Bloch transform
\[
\mathcal{B}_T:L^2_N\mapsto \ell^2\left(\Omega_N, L^2_{\rm per}(0,T)\right)
\]
as defined above satisfies the subharmonic Parseval identify
\begin{equation}\label{e:parseval}
\left<f,g\right>_{L^2_N}=\frac{1}{NT^2}\sum_{\xi\in\Omega_N}\left<\mathcal{B}_T(f)(\xi,\cdot),\mathcal{B}_T(g)(\xi,\cdot)\right>_{L^2(0,T)}
\end{equation}
valid for all $f,g\in L^2_N$.  Finally, we note that if $f\in L^2_{\rm per}(0,T)$ and $g\in L^2_{\rm per}(0,NT)$ then
\[
\mathcal{B}_T(fg)(\xi,x) = f(x)\mathcal{B}_T(g)(\xi,x)
\]
and, additionally, we have the identity
\begin{equation}\label{e:useful_identity}
\left<f,g\right>_{L^2(0,NT)} = \frac{1}{T}\left<f,\mathcal{B}_T(g)(0,\cdot)\right>_{L^2(0,T)}.
\end{equation}
For proofs of \eqref{e:parseval} and \eqref{e:useful_identity}, see \cite[Section 2]{HJP}.

With the above functional analytic tools in hand, we readily see that for our given linearized operator $\L$ and $v\in L^2_N$ we have
\[
\mathcal{B}_T\left(\L v\right)(\xi,x)=\L_\xi\mathcal{B}_T(v)(\xi,x)
\]
so that the operators $\L_\xi$ may be viewed as operator-valued symbols under $\mathcal{B}_T$ acting on $L^2_N$.  Similarly, since the operators $\L$ and $\L_\xi$ are clearly
sectorial on their respective domains, they clearly generate analytic semigroups on $L^2_N$ and $L^2_{\rm per}(0,T)$, respectively, and further their associated semigroups
satisfy
\begin{equation}\label{e:bloch_soln}
\mathcal{B}_T\left(e^{\L t}v\right)(\xi,x)=\left(e^{\L_\xi t}\mathcal{B}_T(v)(\xi,\cdot)\right)(x)~~{\rm and}~~
	e^{\L t}v(x) = \frac{1}{NT}\sum_{\xi\in\Omega_N} e^{i\xi x}e^{\L_\xi t}\mathcal{B}(v)(\xi,x).
\end{equation}
In particular, we see that the $T$-periodic Bloch transform $\mathcal{B}_T$ diagonalizes the periodic coefficient operator differential $\L$ acting on $L^2_N$ in the same way
that the Fourier transform diagonalizes constant-coefficient differential operators acting on $L^2(\RM)$.  This decomposition formula will be used
heavily in our forthcoming linear stability estimates.

\subsection{Diffusive Spectral Stability \& Basic Properties of Bloch Semigroups}

With the above preliminaries, we now establish some important immediate consequences of the diffusive spectral stability assumption. 
As a first result, we describe the unfolding of the Jordan block \eqref{ker1} for the Bloch operators $\L_\xi$ for $|\xi|\ll 1$.  This result
for the KdV/KS system \eqref{e:ks} was established in \cite[Section 3.1]{BJNRZ13}.  See also \cite{JNRZ_Invent} for a more general
version of the same result.

\begin{lemma}[Spectral Preparation]\label{l:spec_prep}
Suppose that $\bar{u}$ is a $T$-periodic traveling wave solution of \eqref{e:ks} which is diffusively spectrally stable and, additionally, satisfies
Assumption \ref{a:non-deg}.  Then the following properties hold.
\begin{itemize}
\item[(i)] For any fixed $\xi_0\in(0,\pi/T)$ there exists a constant $\delta_0>0$ such that
\[
\Re\left(\sigma\left(\L_\xi\right)\right)<-\delta_0
\]
for all $\xi\in[-\pi/T,\pi/T)$ with $|\xi|>\xi_0$.
\item[(ii)] There exist positive constants $\xi_1$ and $\delta_1$ such that for all $|\xi|<\xi_1$, the spectrum of $\L_\xi$ decomposes into two disjoint subsets
\[
\sigma\left(\L_\xi\right)=\sigma_-\left(\L_\xi\right)\bigcup\sigma_0\left(\L_\xi\right)
\]
with the following properties:
\begin{itemize}
\item[(a)] $\Re\left(\sigma_-\left(\L_\xi\right)\right)<-\delta_1$ and $\Re\left(\sigma_c\left(\L_\xi\right)\right)>-\delta_1$.
\item[(b)] The set $\sigma_0\left(\L_\xi\right)$ consists of precisely two eigenvalues, which are analytic in $\xi$ and expand as
\[
\lambda_j(\xi)=-i\xi a_j-d_j\xi^2+\mathcal{O}(\xi^3)~~~j=1,2
\]
for $|\xi|\ll 1$ and some constants $a_j\in\RM$ (distinct) and $d_j>0$.
\item[(c)] The left and right $T$-periodic eigenfunctions $\phi_j(\xi)$ and $\tilde{\phi}(\xi)$ of $\L_\xi$ associated with $\lambda_j(\xi)$ above, normalized so that
\[
\left<\tilde\phi_j(\xi),\phi_k(\xi)\right>_{L^2(0,1)}=i\xi \delta^j_k,~~1\leq j,k\leq 2,
\]
are given as
\begin{align*}
\phi_j(\xi)&=(i\xi)\beta_1^{(j)}(\xi)q_1(\xi)+\beta_{2}^{(j)}(\xi)q_{2}(\xi)\\
\tilde{\phi}_j(\xi)&=\tilde\beta_1^{(j)}(\xi)\tilde q_1(\xi)+(i\xi)\tilde\beta_{2}^{(j)}(\xi)\tilde q_{2}(\xi)
\end{align*}
where the functions $q_j,\tilde{q}_j:[-\xi_0,\xi_0]\to L^2_{\rm per}(0,T)$ are analytic functions such that $\{q_j(\xi,\cdot)\}_{j=1,2}$ and $\{\tilde q_j(\xi,\cdot)\}_{j=1,2}$ are dual
bases of the total eigenspace of $\L_\xi$ associated with the spectrum $\sigma_0(\L_\xi)$ chosen to satisfy
\[
\left(\begin{array}{c}q_1(0,\cdot)\\ q_2(0,\cdot)\end{array}\right) = \left(\begin{array}{c}1\\\bar{u}'\end{array}\right)~~{\rm and}~~
\left(\begin{array}{c}\tilde q_1(0,\cdot)\\ \tilde q_2(0,\cdot)\end{array}\right) = \left(\begin{array}{c}1\\ \psi^{\rm adj}\end{array}\right)
\]
where $\psi^{adj}\in L^2_{\rm per}(0,T)$ is a generalized left eigenfunction satisfying $\L^\dag\psi^{adj}=1$ as well as
\[
\left<\psi^{adj},1\right>_{L^2_1}=0{~~\rm and}~~\left<\psi^{adj},\bar{u}'\right>_{L^2_1}=1,
\]
and where the $\beta_{j},\tilde\beta_j:[-\xi_0,\xi_0]\to\CM$ are analytic functions.
\end{itemize}
\end{itemize}
\end{lemma}

Noting again that the Bloch operators  $\L_\xi$ are sectorial on $L^2_{\rm per}(0,T)$, the spectral results in Lemma \ref{l:spec_prep} immediately imply the following 	
elementary estimates for the Bloch semigroups $e^{\L_\xi t}$.

\begin{proposition}\label{p:semigrp_basicbds}
Suppose that $\bar{u}$ is a $T$-periodic traveling wave solution of \eqref{e:ks} which satisfies the hypothesis of Lemma \ref{l:spec_prep}.  Then the following
properties hold.
\begin{itemize}
\item[(i)] For any fixed $\xi_0\in(0,\pi/T)$, there exist positive constants $C_0$ and $d_0$ such that
\[
\left\|e^{\L_\xi t}f\right\|_{B\left(L^2_{\rm per}(0,T)\right)}\leq C_0 e^{-d_0 t}
\]
is valid for all $t\geq 0$ and all $\xi\in[-\pi/T,\pi/T)$ with $|\xi|>\xi_0$.

\item[(ii)] With $\xi_1$ chosen as in Lemma \ref{l:spec_prep}(ii), 
there exist positive constants $C_1$ and $d_1$ such that for any $|\xi|<\xi_1$, if $\Pi(\xi)$ denotes the (rank-two) spectral projection onto the 
generalized eigenspaces associated to the critical eigenvalues $\{\lambda_j(\xi)\}_{j=1,2}$, then
\[
\left\|e^{\L_\xi t}\left(1-\Pi(\xi)\right)\right\|_{B\left(L^2_{\rm per}(0,T)\right)}\leq C_1 e^{-d_1 t}
\]
is valid for all $t\geq 0$.
\end{itemize}
\end{proposition}

Note that Proposition \ref{p:semigrp_basicbds}(ii) is a natural extension to $|\xi|\ll 1$ of the linear bound \eqref{e:exp_lin_est}, established for $\xi=0$ which
was key in establishing the fixed-$N$ subharmonic stability result Proposition \ref{p:sub_fixedN}.

\section{Subharmonic Linear Estimates}\label{S:lin}

The goal of this section is to obtain decay estimates on the semigroup $e^{\L t}$ acting on $L^2_N$ which are uniform in $N$.  To this end,
we use \eqref{e:bloch_soln} to study the action of $e^{\L t}$ on $L^2_N$ in terms of the associated Bloch operators.  
Naturally, by Lemma \ref{l:spec_prep} we expect for each fixed $N\in\NM$ that the long-time behavior of $e^{\L t}$ is dominated by the generalized kernel of $\L$.
As discussed in the introduction, the main difficulty in obtaining such a result while maintaining uniformity in $N$ is the accumulation of $NT$-periodic
eigenvalues near the origin as $N\to\infty$.  This is overcome through a delicate decomposition of $e^{\L t}$, separating the action into appropriate
critical (i.e., corresponding to spectrum accumulating near the origin) and non-critical frequency components.  This methodology is heavily influenced by the corresponding decomposition
used in the case of localized perturbations of periodic waves: see, for example, \cite{JNRZ_Invent}.  Note, further, that while our basic strategy is
similar to the recent subharmonic analysis of reaction diffusion equations conducted in \cite{JP21}, the new challenge here is the fact that the
linearized operator $\L$ now has two spectral curves which pass through the origin, a reflection of the non-trivial Jordan structure in \eqref{ker1}.

To begin, let $\xi_1\in (0,\pi/T)$ be defined as in Lemma \ref{l:spec_prep} and let $\rho$ be a smooth cutoff function satisfying $\rho(\xi)=1$
for $|\xi|<\xi_1/2$ and $\rho(\xi)=0$ for $|\xi|>\xi_1$.  Given a function $v\in L^2_N$, we then use \eqref{e:bloch_soln} to
decompose the action of $e^{\L t}$ on $v$ into high-frequency and low-frequency components via
\begin{equation}\label{e:shf}
\left\{
\begin{aligned}
e^{Lt}v(x)&=\frac{1}{NT}\sum_{\xi\in\Omega_N}\rho(\xi)e^{i\xi x}e^{L_\xi t}\mathcal{B}_T(v)(\xi,x)
	+\frac{1}{NT}\sum_{\xi\in\Omega_N}(1-\rho(\xi))e^{i\xi x}e^{L_\xi t}\mathcal{B}_T(v)(\xi,x)\\
&=:S_{lf,N}(t)v(x)+S_{hf,N}(t)v(x).
\end{aligned}
\right.
\end{equation}
Using Proposition \ref{p:semigrp_basicbds}(i),  it follows  that there exist constants $C,\eta>0$, both independent of $N$, such that
\[
\max_{\xi\in\Omega_N}(1-\rho(\xi))\|e^{L_\xi t}\|_{B\left(L^2_1\right)}\leq Ce ^{-\eta t}
\]
which, by the subharmonic Parseval identity \eqref{e:parseval}, implies the exponential estimate 
\begin{align*}
\|S_{hf,N}(t)v\|_{L^2_N}^2&=\frac{1}{NT^2}\sum_{\xi\in\Omega_N}\left\|(1-\rho(\xi))e^{L_\xi t}\mathcal{B}_T(v)(\xi,\cdot)\right\|_{L^2_x(0,1)}^2\\
&\leq\frac{1}{NT^2}\sum_{\xi\in\Omega_N}(1-\rho(\xi))^2\|e^{L_\xi t}\|_{B\left(L^2_1\right)}^2\left\|\mathcal{B}_T(v)(\xi,\cdot)\right\|_{L^2_x(0,1)}^2\\
&\leq Ce^{-2\eta t}\left(\frac{1}{NT^2}\sum_{\xi\in\Omega_N}\left\|\mathcal{B}_T(v)(\xi,\cdot)\right\|_{L^2_x(0,1)}^2\right)\\
&=Ce^{-2\eta t}\|v\|_{L^2_N}^2,
\end{align*}
on the high-frequency component of the solution operator.

For the low-frequency component, define for each $|\xi|<\xi_0$ the rank-two spectral projection onto the critical modes of $\L_\xi$ by
\[
\left\{\begin{aligned}
&\Pi(\xi):L^2_1\to\bigoplus_{j=1}^{2}{\rm ker}\left(L_\xi-\lambda_j(\xi)I\right)\\
&\Pi(\xi)f=\sum_{j=1}^{2}q_j(\xi)\left<\tilde{q}_j(\xi,\cdot),f\right>_{L^2(0,1)}.
\end{aligned}
\right.
\]
Note we also have the alternate representation formula
\[
\Pi(\xi)f=\frac{1}{i\xi}\sum_{j=1}^{2}\phi_j(\xi)\left<\tilde\phi_j(\xi,\cdot),f\right>_{L^2(0,1)},
\]
which more explicitly demonstrates the singularity as $\xi\to 0$ of the eigenprojection.
Using this, the low-frequency operator $S_{lf,N}(t)$ can
be further decomposed into the contribution from the critical modes near $(\lambda,\xi)=(0,0)$ and the contribution from the
low-frequency spectrum bounded away from $\xi=0$ via
\begin{equation}\label{e:slf}
\left\{
\begin{aligned}
S_{lf,N}(t)v(x)&=\frac{1}{NT}\sum_{\xi\in\Omega_N}\rho(\xi)e^{i\xi x}e^{L_\xi t}\Pi(\xi)\mathcal{B}_T(v)(\xi,x)\\
&\quad\qquad+\frac{1}{NT}\sum_{\xi\in\Omega_N}\rho(\xi)e^{i\xi x}e^{L_\xi t}\left(1-\Pi(\xi)\right)\mathcal{B}_T(v)(\xi,x)\\
&=:S_{c,N}(t)v(x)+\widetilde{S}_{lf,N}(t)v(x).
\end{aligned}
\right.
\end{equation}
As above, by possibly choosing $\eta>0$ smaller, Proposition \ref{p:semigrp_basicbds}(ii) implies that there exists a constant $C>0$ such that for each $N\in\NM$ we have
\[
\left\|e^{L_\xi t}(1-\Pi(\xi))f\right\|_{L^2_1}\leq Ce^{-\eta t}\|f\|_{L^2_1}
\]
for each $f\in L^2_1$, and hence another application of Parseval's identity \eqref{e:parseval} yields
\[
\left\|\widetilde{S}_{lf,N}(t)v\right\|_{L^2_N}\leq Ce^{-\eta t}\|v\|_{L^2_N}.
\]

For the critical component, we decompose $S_{c,N}(t)$ further as
\begin{align*}
S_{c,N}(t)v(x)&=\frac{1}{NT}e^{L_0t}\Pi(0)\mathcal{B}_T(v)(0,x)\\
&\qquad+\frac{1}{NT}\sum_{\xi\in\Omega_N\setminus\{0\}}\rho(\xi)e^{i\xi x} 
	\sum_{j=1}^{2}\frac{1}{i\xi}e^{\lambda_j(\xi) t}\phi_j(\xi,x)\left<\tilde\phi_j(\xi,\cdot),\mathcal{B}_T(v)(\xi,\cdot)\right>_{L^2_1}
\end{align*}
For the non-zero frequencies, noting that Lemma \ref{l:spec_prep} implies that the  eigenfunctions $\left\{\phi_j(\xi)\right\}_{j=1}^2$ of $L_\xi$ expand as 
\begin{align*}
\phi_j(\xi,x)&=\beta_{2}^{(j)}(\xi)\left[\bar{u}'(x)+\left(q_{2}(\xi,x)-\bar{u}'(x)\right)\right]+(i\xi)\beta_1^{(j)}(\xi)q_{1}(\xi,x)\\
&=\beta_{2}^{(j)}(\xi)\bar{u}'(x)+\mathcal{O}(\xi),
\end{align*}
for $0<|\xi|\ll 1$ suggests the decomposition
\begin{equation}\label{e:sc}
\left\{
\begin{aligned}
S_{c,N}(t)v(x)&=\frac{1}{NT}e^{L_0t}\Pi(0)\mathcal{B}_T(v)(0,x)\\
&\quad+\bar{u}'(x)\left(\frac{1}{NT}\sum_{\xi\in\Omega_N\setminus\{0\}}\rho(\xi)e^{i\xi x} 
		\sum_{j=1}^{2}\frac{1}{i\xi}e^{\lambda_j(\xi)t}\beta_{2}^{(j)}(\xi)\left<\tilde\phi_j(\xi,\cdot),\mathcal{B}_T(v)(\xi,\cdot)\right>_{L^2_1}\right)\\
&\qquad+	\frac{1}{NT}\sum_{\xi\in\Omega_N\setminus\{0\}}\rho(\xi)e^{i\xi x} 
				\sum_{j=1}^{2}e^{\lambda_j(\xi) t}\frac{\phi_j(\xi,x)-\beta_{2}^{(j)}(\xi)\bar{u}'(x)}{i\xi}\left<\tilde\phi_j(\xi,\cdot),\mathcal{B}_T(v)(\xi,\cdot)\right>_{L^2_1}\\				
&=:\frac{1}{NT}e^{L_0t}\Pi(0)\mathcal{B}_T(v)(0,x)+\bar{u}'(x)s_{p,N}(t)v(x)+\widetilde{S}_{c,N}(t)v(x).
\end{aligned}
\right.
\end{equation}
Additionally, note the $\xi=0$ term above can be expressed as 
\[
e^{\mathcal L_0t}\Pi(0)\mathcal{B}_T(v)(0,x)=
e^{\mathcal L_0 t} \left<1,B_1(v)(0,\cdot)\right>_{L^2_1}+\bar{u}'(x)\left<\psi^{adj}, B_1(v)(0,\cdot)\right>_{L^2_1}.
\]
Observing that\footnote{Indeed, observe that $\frac{d}{dt}\left(1-t\bar{u}'\right) = -\bar{u}'=\L_0(1)=\L_0(1-t\bar{u}')$ and clearly the associated initial condition is satisfied.}
\[
e^{\mathcal L_0t}1=1-t\bar{u}'
\]
and noting that \eqref{e:useful_identity} implies the identities
\[
\int_0^T\mathcal{B}_T(v)(0,x)dx =T\int_0^{NT}v(x)dx~~{\rm and}~~\left<\psi^{adj},\mathcal{B}_T(v)(0,\cdot)\right>_{L^2_1}=T\left<\psi^{adj},v\right>_{L^2_N}
\]
we find the representation 
\begin{align*}
\frac{1}{NT}e^{\mathcal L_0t}\Pi(0)\mathcal{B}_T(v)(0,x)&=\frac{1}{N}\int_0^{NT}v(z)dz\\
&\quad			+\frac{1}{N}\bar{u}'(x)\left(\left<\psi^{adj}, v\right>_{L^2_N}-t\int_0^{NT}v(z)dz\right).
\end{align*}
Observe that the above representation of the $\xi=0$ contribution of the solution operator clearly demonstrates the expected linear instability\footnote{This is naturally
expected due to the Jordan block at the $\lambda=0$ eigenvalue of $\L_0$.} of the underlying wave $\bar{u}$.
In the forthcoming analysis nonlinear analysis, this linear instability will be compensated by allowing
for Galilean boosts of the underlying wave.

Taken together, it follows that the linear solution operator $e^{\L t}$ acting on $L^2_N$ can be decomposed as
\begin{equation}\label{e:lin_decomp_final}
\left\{
\begin{aligned}
e^{\L t}v(x)&= \bar{u}'(x)\left(\frac{1}{N}\left<\psi^{adj}, v\right>_{L^2_N}-\frac{t}{N}\int_0^{NT}v(z)dz + s_{p,N}(t)v(x)\right)\\
&\qquad +\frac{1}{N}\int_0^{NT}v(z)dz + \widetilde{S}_N(t)v(x),
\end{aligned}
\right.
\end{equation}
where here
\[
\widetilde{S}_N(t) := S_{hf,N}(t) + \widetilde{S}_{lf,N}(t)+\widetilde{S}_{c,N}(t)
\]
and the operators $S_{hf,N}(t)$ and $\widetilde{S}_{lf,N}(t)$  are defined in \eqref{e:shf} and \eqref{e:slf}, respectively, and  $\widetilde{S}_{c,N}(t)$ and $s_{p,N}(t)$ are defined in \eqref{e:sc}.
With this decomposition in  hand, we now establish temporal estimates on the above which are uniform in $N\in\NM$.

\begin{proposition}[Linear Estimates]\label{p:lin_est}
Suppose that $\bar{u}$ is a $T$-periodic diffusively spectrally stable traveling wave solution of \eqref{e:ks} with wave speed $c$.  Given any $M\in\NM$ there 
exists a constant $C>0$ such that for all $N\in\NM$, $t\geq 0$ and all $1	\leq \ell,m,r\leq M$ we have
\begin{equation}\label{e:lin_bd1}
\left\|\partial_x^\ell\partial_t^m s_{p,N}(t)v\right\|_{L^2_N}\leq C (1+t)^{1/4-(\ell+m)/2}\|v\|_{L^2_N}
\end{equation}
and
\begin{equation}\label{e:lin_bd2}
\left\|\partial_x^\ell\partial_t^ms_{p,N}(t)\partial_x^r v\right\|_{L^2_N}\leq C (1+t)^{-1/4-(\ell+m)/2}\|v\|_{L^2_N}.
\end{equation}
Further, there exist constants $C,\eta>0$ such that for all $t\geq 0$ $N\in\NM$, and all  $0\leq \ell,m,r-1\leq M$ we have
\begin{equation}\label{e:lin_bd3}
\left\|\partial_x^\ell\partial_t^m\widetilde{S}_N(t)v\right\|_{L^2_N}\leq C\left((1+t)^{-1/4}\|v\|_{L^1_N} + e^{-\eta t}\|v\|_{H^{\ell+4m}_N}\right)
\end{equation}
and
\begin{equation}\label{e:lin_bd4}
\left\|\partial_x^\ell\partial_t^m\widetilde{S}_N(t)\partial_x^r v\right\|_{L^2_N}\leq C\left((1+t)^{-3/4}\|v\|_{L^1_N} + e^{-\eta t}\|v\|_{H^{r+\ell+4m}_N}\right).
\end{equation}
Finally, there exists a constant $C>0$, independent of $N$, such that for all $t\geq 0$ we have
\begin{equation}\label{e:lin_bd5}
\left\|s_{p,N}(t)v\right\|_{L^\infty_N}\leq C \|v\|_{L^1_N\cap L^2_N}~~{\rm and}~~\left\|s_{p,N}(t)\partial_x v\right\|_{L^\infty_N}\leq C(1+t)^{-1/2} \|v\|_{L^1_N\cap L^2_N}.
\end{equation}
\end{proposition}

\begin{proof}
We begin by establishing the $L^2_N$ estimates above.  To this end, first note that the definition of $\mathcal{B}_T$ implies that
\begin{align*}
\left<\tilde\phi_j(\xi,\cdot),\mathcal{B}_T(v)(\xi,\cdot)\right>_{L^2_1}&=\int_0^T\overline{\tilde\phi_j(\xi,x)}\sum_{\ell\in\ZM}e^{2\pi i\ell x/T}\hat{v}(\xi+2\pi\ell/T)\\
&=\sum_{\ell\in\ZM}\hat{v}(\xi+2\pi\ell/T)\int_0^T e^{2\pi i\ell x/T}\overline{\tilde\phi_j(\xi,x)}dx\\
&=\sum_{\ell\in\ZM}\hat{v}(\xi+2\pi\ell/T)\overline{\widehat{\tilde\phi_j}(\xi,2\pi\ell/T)}.
\end{align*}
Since  $\|\hat{g}\|_{L^\infty(\RM)}\leq \|g\|_{L^1_N}$ by \eqref{e:Fourier}, an application of Cauchy-Schwarz 
implies the existence of a constant $C>0$, independent of $N$, such that
\begin{align*}
\rho(\xi)\left|\left<\tilde\phi_j(\xi,\cdot),\mathcal{B}_T(v)(\xi,\cdot)\right>_{L^2_1}\right|^2&\leq	
		\rho(\xi)\|v\|_{L^1_N}^2\left(\sum_{\ell\in\ZM}(1+|\ell|^2)^{1/2}\left|\overline{\widehat{\tilde\phi_j}(\xi,2\pi\ell/T)}\right|(1+|\ell|^2)^{-1/2}\right)\\
&\leq C\|v\|_{L^1_N}^2\sup_{\xi\in[-\pi/T,\pi/T)}\left(\rho(\xi)\|\tilde\phi_j(\xi,\cdot)\|_{H^1_{\rm per}(0,T)}^2\right),
\end{align*}
valid for all $\xi\in\Omega_N$.  Using the subharmonic Parseval identity \eqref{e:parseval}, along with Lemma \ref{l:spec_prep}, we find that 
\begin{equation}\label{e:lin_bd_1}
\begin{aligned}
\left\|\partial_x^\ell\partial_t^m s_{p,N}(t)v\right\|_{L^2_N}^2\leq&
	\frac{1}{NT^2}\sum_{j=1}^2\sum_{\xi\in\Omega_N\setminus\{0\}}\left\|\rho(\xi)(i\xi)^{\ell-1}(\lambda_j(\xi))^m e^{\lambda_j(\xi)t}\left<\tilde\phi_j(\xi,\cdot),\mathcal{B}_T(v)(\xi,\cdot)\right>_{L^2_1}\right\|_{L^2_1}^2\\
&\leq C\|v\|_{L^1_N}^2\left(\frac{1}{NT}\sum_{\xi\in\Omega_N\setminus\{0\}}|\xi|^{2(\ell+m-1)}e^{2d\xi^2t}\right),
\end{aligned}
\end{equation}
where here $C,d>0$ are constants which are independent of $N$.  By Similar considerations, we find
\begin{equation}\label{e:lin_bd_2}
\left\|\partial_x^\ell\partial_t^m\widetilde{S}_N(t)v\right\|_{L^2_N}^2\leq Ce^{-2\eta t}\|v\|_{L^2_N}^2
	+C\|v\|_{L^1_N}^2\left(\frac{1}{NT}\sum_{\xi\in\Omega_N\setminus\{0\}}|\xi|^{2(\ell+m)}e^{2d\xi^2t}\right)
\end{equation}
where, again, the constants $C,d>0$ are independent of $N$.

Continuing, note that integration by parts yields the identity
\[
\mathcal{B }_T(\d_x g)(\xi,x) = (\d_x+i\xi)\mathcal{B}_T(g)(\xi,x).
\]
%
%
It follows that for each $1\leq j\leq 2$ and integer $r\geq 1$ we have
\begin{equation}\label{e:extra_der}
\begin{aligned}
&\left<\tilde\phi_j(\xi,\cdot),\mathcal{B}_T(\partial_x^r g)(\xi,\cdot)\right>_{L^2_1}=\left<\tilde\phi_j(\xi,\cdot),(\partial_x+i\xi)^r\mathcal{B}_T(g)(\xi,\cdot)\right>_{L^2_1}\\
&\qquad\qquad=\sum_{r'=0}^r\left(\begin{array}{c}r\\r'\end{array}\right)(i\xi)^{r'}\left<\tilde\phi_j(\xi,\cdot),\partial_x^{r-r'}\mathcal{B}_T(g)(\xi,\cdot)\right>_{L^2_1}\\
&\qquad\qquad=(-1)^r\left<\partial_x^r\tilde\phi_j(\xi,\cdot),\mathcal{B}_T( g)(\xi,\cdot)\right>_{L^2_1}
		 +\sum_{r'=1}^r\left(\begin{array}{c}r\\r'\end{array}\right)(i\xi)^{r'}\left<\tilde\phi_j(\xi,\cdot),\partial_x^{r-r'}\mathcal{B}_T(g)(\xi,\cdot)\right>_{L^2_1},
\end{aligned}
\end{equation}
where the last equality follows by integrating by parts in the $r'=0$ term.  Since Lemma \ref{l:spec_prep} implies that $\tilde\phi_j(0,x)$ is a constant vector for each $\xi$, we have
\[
\partial_x\tilde\phi_j(\xi,x)=\xi\partial_x\left(\frac{\tilde\phi_j(\xi,x)-\tilde\phi_j(0,x)}{\xi}\right),
\]
which is clearly $\mathcal{O}(\xi)$ by the analytic dependence of $\tilde\phi_j(\xi,x)$ on $\xi$.  Taken together, it follows that for each $r\geq 1$ we have
\begin{equation}\label{e:lin_bd_3}
\begin{aligned}
\left\|\partial_x^\ell\partial_t^m s_{p,N}(t)\partial_x^r v\right\|_{L^2_N}^2\leq&
C\|v\|_{L^1_N}^2\left(\frac{1}{NT}\sum_{\xi\in\Omega_N\setminus\{0\}}|\xi|^{2(\ell+m)}e^{2d\xi^2t}\right)
\end{aligned}
\end{equation}
and, similarly,
\begin{equation}\label{e:lin_bd_4}
\left\|\partial_x^\ell\partial_t^m\widetilde{S}_N(t)\partial_x^r v\right\|_{L^2_N}^2\leq Ce^{-2\eta t}\|v\|_{H^r_N}^2
	+C\|v\|_{L^1_N}^2\left(\frac{1}{NT}\sum_{\xi\in\Omega_N\setminus\{0\}}|\xi|^{2(\ell+m+1)}e^{2d\xi^2t}\right)
\end{equation}

To complete the proof of the $L^2_N$ bounds, it remains to provide bounds on the discrete, $N$-dependent sums in \eqref{e:lin_bd_1}-\eqref{e:lin_bd_4}. 
These uniform bounds come by directly applying Lemma A.1 in \cite{JP21}, which states that for any integer $\omega\geq 0$, there exists a constant $C>0$, independent of $N$, 
such that\footnote{As motivation, observe that the sum $\frac{1}{NT}\sum_{\xi\in\Omega_N}\xi^{2\omega}e^{-2d\xi^2t}$
can be considered as a Riemann sum approximation for the integral $\int_{-\pi/T}^{\pi/T}\xi^{2\omega}e^{-2d\xi^2t}d\xi$, which exhibits the 
stated temporal decay via a routine scaling argument.}
\[
\frac{1}{N}\sum_{\xi\in\Omega_N}\xi^{2\omega}e^{-2d\xi^2t}\leq C(1+t)^{-\omega-1/2}.
\]
Applying this result to \eqref{e:lin_bd_1}-\eqref{e:lin_bd_4} with the appropriate values of $\omega$ establishes
the $L^2_N$ estimates stated in \eqref{e:lin_bd1}-\eqref{e:lin_bd4}.

It remains to establish the $L^\infty$ estimates stated in \eqref{e:lin_bd5}.  To this end, note that by similar estimates
as above, the $L^\infty$ difference between $s_{p,N}(t)v(x)$ and the function
\begin{equation}\label{e:approx1}
\sum_{j=1}^2\frac{1}{NT}\sum_{k\in\ZM\setminus\{0\}}e^{i\xi_k}\frac{e^{-d_j\xi_k^2t}}{i\xi_k}\beta_2^j(0)\left<\widetilde\phi_j(0,\cdot),\mathcal{B}_T(v)(\xi_k,\cdot)\right>_{L^2_1}
=:s^{(1)}_{p,N}(t)v(x)+s^{(2)}_{p,N}(t)v(x)
\end{equation}
is controlled by $(1+t)^{-1/2}\|v\|_{L^1_N}$, where here $\xi_k:=2\pi k/NT$ for each $k\in\ZM$.  Further, since for each $j=1,2$ the function $\widetilde\phi_j(0,x)$ is actually 
a constant equal to $\tilde\nu_j:=\left(\tilde\beta_1^{(j)}(0),\tilde\beta_2^{(j)}(0)\right)$, we note that
\[
\left<\widetilde\phi_j(0,\cdot),\mathcal{B}_T(v)(\xi_k,\cdot)\right>_{L^2_1}=\tilde\nu_j\cdot\hat{v}(\xi_k)
\]
and hence, since  \eqref{e:approx1} is in the form of a Fourier series itself, for each $j=1,2$ the function $s^{(j)}_{p,N}(t)v(x)$ 
can be recognized as the convolution of $\tilde\nu_j\cdot v$ with
\begin{equation}\label{e:conv}
\frac{1}{NT}\sum_{k\in\ZM\setminus\{0\}}e^{i\xi_k}\frac{e^{-d_j\xi_k^2t}}{i\xi_k}\beta_2^j(0).
\end{equation}
Recognizing the above sums as $T$-periodic Fourier series representations for Error functions, which are clearly bounded in $L^\infty$,
this establishes the first $L^\infty$ bound in \eqref{e:lin_bd5}.  The second bound in \eqref{e:lin_bd5} now follows by using precisely 
the same procedure as above, while noting that \eqref{e:extra_der} implies the extra derivative on $v$ yields an extra $\mathcal{O}(\xi)$
factor in the sum \eqref{e:conv} which, in turn, yields the additional $(1+t)^{-1/2}$ decay.
\end{proof}

Before continuing to our nonlinear analysis, we first provide an interpretation of the above decomposition of the linear solution operator.  To this end, suppose
$\bar{u}$ is a $T$-periodic diffusively spectrally stable periodic traveling wave solution of \eqref{e:ks}, and suppose that $u(x,t)$ is a solution (in same co-moving frame)
with initial data $u(x,0)=\phi(x)+\eps v(x)$ with $|\eps|\ll 1$ and $v\in L^1_N\cap L^2_N$.  From the decomposition \eqref{e:lin_decomp_final} and the linear estimates
in Proposition \ref{p:lin_est}, it is then natural to suspect that
\begin{align*}
u(x,t)&\approx \bar{u}(x)+\eps e^{\L t}v(x)\\
&\approx \bar{u}(x)+\eps\bar{u}'(x)\left(\frac{1}{N}\left<\psi^{adj},v\right>_{L^2_N}+s_{p,N}(t)v(x)-\frac{t}{N}\int_0^{NT}v(z)dz\right)+\frac{\eps}{N}\int_0^{NT}v(z)dz\\
&\approx \bar{u}\left(x+\eps\left(\frac{1}{N}\left<\psi^{adj},v\right>_{L^2_N}+s_{p,N}(t)v(x)\right)-\frac{\eps t}{N}\int_0^{NT}v(z)dz\right)+\frac{\eps}{N}\int_0^{NT}v(z)dz,
\end{align*}
which is a (small, via the $L^\infty$ bounds in \eqref{e:lin_bd5}) spatio-temporal phase modulation of the background wave $\bar{u}$ together with an (identical) wave-speed and mass correction.  
In particular, recalling that solutions of the KdV/KS equation \eqref{e:ks} are invariant under the Galilean transformation \eqref{galilean}
the above suggests that the initially nearby wave $u(x,t)$, up to  spatio-temporal phase modulation, will asymptotically
approach a member of the Galilean family associated to the background wave $\bar{u}$.  In the next section, we verify this intuition.

\section{Uniform Nonlinear Asymptotic Stability} \label{S:nonlinear}

Our goal is to now use the linear decomposition \eqref{e:lin_decomp_final} in order to establish the proof of Theorem \ref{T:main}.  
As discussed above, 
a small subharmonic perturbation of a $T$-periodic, diffusively spectrally stable periodic traveling wave solution $\bar{u}$ of \eqref{e:ks}
will, for long time, behave like a coupled temporal mass modulation and spatio-temporal phase modulation of background wave $\bar{u}$.

\subsection{Nonlinear Decomposition, Perturbation Equations \& Damping}

Suppose that $\bar{u}$ is a $T$-periodic diffusively spectrally stable periodic traveling wave solution of \eqref{e:ks}.  Motivated
by the decomposition \eqref{e:lin_decomp_final}, we develop a decomposition of nonlinear, subharmonic perturbations of the background wave $\bar{u}$
which accounts for the phase and mass modulations predicted by the linear theory.  

To this end, suppose that ${u}(x,t)$ is a solution (in the same co-moving frame) with initial data ${u}(x,0)\in L^1_N\cap L^2_N$
which is close (in $L^2_N$) to $\bar{u}$.  By conservation of mass, we note that
\[
\int_0^{NT}{u}(x,t)dx = \int_0^{NT}{u}(x,0)dx
\]
for all $t\geq 0$ for which it is defined.  In particular, unless the initial data ${u}(x,0)$ has the same mass as the underlying wave $\bar{u}$,
we should not expect asymptotic convergence of ${u}(x,t)$ to $\bar{u}$.  Recalling, however, that solutions of \eqref{e:ks} obey
the Galilean invariance \eqref{galilean}, it is natural to suspect that the added mass from the initial perturbation may induce a Galilean boost
of the background wave $\bar{u}$.  Specifically, given such a ${u}(x,0)\in L^1_N\cap L^2_N$ and defining 
\[
\delta M:=\frac{1}{NT}\int_0^{NT}\left({u}(z,0)-\bar{u}(z)\right)dz,
\]
we expect that the associated nearby solution ${u}(x,t)$ should satisfy
\[
{u}(x,t)\approx \bar{u}\left(x-\left(\delta M\right) t\right)+\delta M
\]
for $t\gg 1$. Note that for each fixed $N$ this is exactly the long-time dynamics predicted by Proposition \ref{p:sub_fixedN}
As discussed in the introduction, however, the result of Proposition \ref{p:sub_fixedN} lacks uniformity with respect to the period of the perturbations
in both the rate of decay of perturbations and the allowable size of initial perturbations.  In this section, we use the linear results in 
Section \ref{S:lin} above, which were designed specifically to be uniform in the perturbation's period, to establish our main result Theorem \ref{T:main}.

To this end, let ${u}(x,t)$ be a solution (in the same co-moving frame as $\bar{u}$) with initial data ${u}(x,0)\in L^1_N\cap L^2_N$
which is close (in $L^2_N$) to $\bar{u}$ and consider a nonlinear perturbation of the form
\begin{equation}\label{e:pert}
    v(x,t) := {u}\left(x + \left(\delta M\right)t - \frac{1}{N}\gamma(t) - \psi(x,t), t\right) - \delta M - \bar{u}(x)
\end{equation}
where $\gamma:\RM_+\to\RM$ and $\psi:\RM\times\RM_+\to\RM$ are functions to be determined later\footnote{Note though that we assume $\psi(\cdot,t)\in L^2_N$ for each $t\geq 0$.}.
Note, in particular, that integrating \eqref{e:pert} over $[0,NT]$ gives
\[
\int_0^{NT}{u}(z,t)dz = \int_0^{NT}\bar{u}(z)dz+NT\left(\delta M\right) + \int_0^{NT}v(z,t)dz
\]
which, since the integral of ${u}(\cdot,t)$ over $[0,NT]$ is conserved by the flow of \eqref{e:ks}, implies that
\begin{equation}\label{e:pert_mass}
\int_0^{NT}v(z,t)dz = 0
\end{equation}
for all $t\geq 0$ for which it is defined by the choice of $\delta M$ above.  

With the above decomposition in hand, we next derive equations that must be satisfied by the perturbation $v$ and the modulation functions
$\gamma$ and $\psi$.

\begin{proposition}\label{p:nonlin_pert}
The triple $(v,\gamma,\psi)$ satisfies
\begin{equation}\label{nonlin_pert}
(\d_t - \mathcal{L})\left((1-\psi_x)v + \frac{1}{N}\gamma \bar{u}' + \psi\bar{u}'\right) = \d_x\mathcal{N}, 
\end{equation}
where
\[
\d_x\mathcal{N} = \d_x\mathcal{Q} + \d_x\mathcal{R} + \mathcal{L}(\psi_x v),\quad \mathcal{Q} = -\frac{1}{2}v^2,
\]
and 
\begin{align*}
    \mathcal{R} &= -\psi_t v - \frac{1}{N}\gamma'v - \eps\left[\d_x\left(\frac{\psi_x}{1-\psi_x}v_x\right) + \frac{\psi_x}{1-\psi_x}\d_x\left(\frac{1}{1-\psi_x}v_x\right)\right] - \delta\left(\frac{\psi_x}{1-\psi_x}v_x\right)\\
    &\qquad - \delta\left[\d_x^2\left(\frac{\psi_x}{1-\psi_x}v_x\right) + \d_x\left(\frac{\psi_x}{1-\psi_x}\d_x\left(\frac{1}{1-\psi_x}v_x\right)\right) + \frac{\psi_x}{1-\psi_x}\d_x\left(\frac{1}{1-\psi_x}\d_x\left(\frac{1}{1-\psi_x}v_x\right)\right)\right]\\
    &\qquad - \eps\left[\frac{\psi_x}{1-\psi_x}\d_x\left(\frac{\psi_x}{1-\psi_x}\bar{u}'\right) + \d_x\left(\frac{\psi_x^2}{1-\psi_x}\bar{u}'\right) + \frac{\psi_x^2}{1-\psi_x}\bar{u}''\right] - \delta\left(\frac{\psi_x^2}{1-\psi_x}\bar{u}'\right)\\
    &\qquad - \delta\left[\frac{1}{1-\psi_x}\d_x\left(\frac{\psi_x}{1-\psi_x}\d_x\left(\frac{\psi_x}{1-\psi_x}\bar{u}'\right)\right) + \frac{\psi_x}{1-\psi_x}\d_x^2\left(\frac{\psi_x}{1-\psi_x}\bar{u}'\right) + \frac{\psi_x}{1-\psi_x}\d_x\left(\frac{\psi_x}{1-\psi_x}\bar{u}''\right)\right.\\
    &\qquad+ \left. \d_x^2\left(\frac{\psi_x^2}{1-\psi_x}\bar{u}'\right) + \d_x\left(\frac{\psi_x^2}{1-\psi_x}\bar{u}''\right) + \frac{\psi_x^2}{1-\psi_x}\bar{u}'''\right].
\end{align*}
\end{proposition}

\begin{proof}
The proof of the above is a relatively routine, yet long calculation.  For completeness, we present the details in Appendix \ref{A:pert_eqns}.
\end{proof}

Our goal is now to obtain a closed nonlinear iteration scheme by integrating \eqref{nonlin_pert} and exploiting the decomposition of the linear 
solution operator $e^{\L t}$ provided in \eqref{e:lin_decomp_final}.  To motivate this, we first provide an informal description of how to identify
the modulation functions $\gamma$ and $\psi$.  To this end, note that using Duhamel's formula we can rewrite \eqref{nonlin_pert} as the equivalent
implicit integral equation
\[
\left(1-\psi_x(x,t)\right)v(x,t) + \frac{1}{N}\gamma(t) \bar{u}'(x) + \psi(x,t)\bar{u}'(x) = e^{\L t}v(x,0)+\int_0^t e^{\L (t-s)}\partial_x\mathcal{N}(x,s)ds,
\]
where here we have taken the initial data $\gamma(0)=0$, $\psi(\cdot,0)\equiv 0$ and $v(x,0) ={u}(x,0)-\delta M-\bar{u}(x)$.  
Recalling that \eqref{e:lin_decomp_final} implies the linear solution operator can be decomposed as
\[
\begin{aligned}
e^{\mathcal Lt}v(x) &=  \bar{u}'(x)\left[\underbrace{\frac{1}{N}\left<\psi^{adj},v\right>_{L^2_N}+s_{p,N}(t)v(x)}_{\rm phase~modulation}
	-t\underbrace{\left(\frac{1}{N}\int_0^{NT}v(z)dz\right)}_{\rm wave~speed~correction}\right]\\
&\quad+\underbrace{\frac{1}{N}\int_0^{NT}v(z)dz}_{\rm mass~modulation} +\underbrace{\widetilde{S}_N(t)v(x)}_{\rm faster~decaying~residual},
\end{aligned}
\]
it follows that we can remove the principle (i.e., slowest decaying) part
of the nonlinear perturbation by implicitly defining
\begin{equation}\label{e:mod_longtime}
\left\{\begin{aligned}
&\gamma(t)\sim \left<\psi^{adj},v(\cdot,0)\right>_{L^2_N}+\int_0^t\left<\psi^{adj},\partial_x\mathcal{N}(\cdot,s)\right>_{L^2_N}ds\\
&\psi(x,t)\sim s_{p,N}(t)v(x,0) + \int_0^t s_{p,N}(t-s)\partial_x\mathcal{N}(x,s)ds,
\end{aligned}\right.
\end{equation}
where here $\sim$ indicates equality for $t\geq 1$.  This choice then yields the implicit description
\begin{equation}\label{e:pert_longtime}
v(x,t)\sim\psi_x(x,t)v(x,t)+\widetilde{S}_N(t)v(x,0) + \int_0^t\widetilde{S}_N(t-s)\partial_x\mathcal{N}(x,s)ds
\end{equation}
involving only the faster decaying residual component of the linear solution operator.  

To make the above  choices consistent (in short time) with the prescribed
initial data, we simply interpolate between the initial data and the long-time choices prescribed in \eqref{e:mod_longtime}-\eqref{e:pert_longtime} above.  To this end,
let $\chi(t)$ be a smooth cutoff function that is zero for $t\leq 1/2$ and one for $t\geq 1$, and define the modulation functions $\gamma$ and $\psi$ implicitly
for all $t\geq 0$ as
\begin{equation}\label{e:mod_def}
\left\{\begin{aligned}
&\gamma(t)=\chi(t)\left[ \left<\psi^{adj},v(\cdot,0)\right>_{L^2_N}+\int_0^t\left<\psi^{adj},\partial_x\mathcal{N}(\cdot,s)\right>_{L^2_N}ds\right]\\
&\psi(x,t)=\chi(t)\left[ s_{p,N}(t)v(x,0) + \int_0^t s_{p,N}(t-s)\partial_x\mathcal{N}(x,s)ds\right],
\end{aligned}\right.
\end{equation}
which now leaves the implicit description
\begin{equation}\label{e:pert_def}
\begin{aligned}
v(x,t)&=\left(1-\chi(t)\right)\left[e^{\L t}v(x,0)+\int_0^te^{\L (t-s)}\partial_x\mathcal{N}(s)ds\right]\\
&\qquad +\chi(t)\left( \psi_x(x,t)v(x,t)+\widetilde{S}_N(t)v(x,0) + \int_0^t\widetilde{S}_N(t-s)\partial_x\mathcal{N}(x,s)ds\right)
\end{aligned}
\end{equation}
to be satisfied in $L^2_N$ for all $t\geq 0$.  Observe, however, that there is an inherent loss of derivatives associated with the system \eqref{e:mod_def}-\eqref{e:pert_def}.  For
example, attempting to control the $L^2_N$ norm of $v(\cdot,t)$ via the implicit equation \eqref{e:pert_def} requires control of at least the $H^4_N$ derivative
of $v$ (as well as various derivatives of $\gamma$ and $\psi$).  This loss of derivatives may be compensated by the fact that the damping in \eqref{e:ks} corresponds
to the highest-order spatial derivative, which allows high Sobolev norms of $v$ to be slaved to low Sobolev norms of $v$ plus sufficient control
on the modulation functions $\gamma$ and $\psi$.  This is the content of the following technical result.

\begin{proposition}[Nonlinear Damping]\label{p:damping}
Fix $N\in\NM$ and suppose that the nonlinear perturbation $v$ defined in \eqref{e:pert} satisfies $v(\cdot,0)\in H^5_N$.  Then there exist positive constants
$\theta$, $C$ and $\eps_0$ (independent of $N$) such that if $v$, $\psi$ and $\gamma$ solve the system \eqref{e:mod_def}-\eqref{e:pert_def} on $[0,\tau]$
for some $\tau>0$ and
\[
\sup_{t\in[0,\tau]}\left(\left\|(v,\psi_x)\right\|_{H^5_N}+\left\|\psi_t(t)\right\|_{H^4_N}+\left|\gamma'(t)\right|\right)\leq\eps_0
\]
then, for all $0\leq t\leq \tau$ we have
\begin{align*}
\|v(t)\|_{H^5_N}^2&\leq Ce^{-\theta t}\|v(0)\|_{H^5_N}^2\\
&\quad +C\int_0^t e^{-\theta (t-s)}\left(\|v(s)\|_{L^2_N}^2+\|\psi_x(s)\|_{H^6_N}^2+\|\psi_t(s)\|_{H^3_N}^2+|\gamma_t(s)|^2\right)ds.
\end{align*}
\end{proposition}

The proof of the above estimate is by now standard, and can be found (for perturbations in $L^2(\RM)$) in \cite[Proposition 3.4]{BJNRZ13}.
The result in \cite{BJNRZ13} directly adapts to the current situation and is hence omitted.

\subsection{Nonlinear Iteration}

With the above nonlinear preliminaries, we now complete the proof of Theorem \ref{T:main}.  
Associated to the solution $(v,\gamma_t,\psi_x,\psi_t)$ of \eqref{e:mod_def}-\eqref{e:pert}
we define, so long as it is finite, the function
\[
\zeta(t):=\sup_{0\leq s\leq t}\left(\|v(s)\|_{H^5_N}^2+\|\psi_t(s)\|_{H^5_N}^2+\|\psi_x(s)\|_{H^6_N}^2+|\gamma_t(s)|\right)^{1/2}(1+s)^{1/4}.
\]
Using the linear estimates from Section \ref{S:lin} as well as the above nonlinear preparations, we now establish an estimate on $\zeta$ that will establish both the global existence 
of the nearby solution ${u}$ and the modulation functions, but will also establish our main stability result.

\begin{proposition}\label{p:induct}
Under the hypotheses of Theorem \ref{T:main}, there exist positive constants $C,\eps>0$, both independent of $N$, such that if $v(\cdot,0)$ satisfies
\[
E_0:=\|v(\cdot,0)\|_{L^1_N\cap H^5_N}\leq\eps~~{\rm and}~~\zeta(\tau)\leq\eps
\]
for some $\tau>0$, then we have
\[
\zeta(t)\leq C\left(E_0+\zeta(t)^2\right)
\]
for all $0\leq t\leq \tau$.
\end{proposition}

\begin{proof}
From Proposition \ref{p:nonlin_pert}, we note that
\begin{equation}\label{e:Nbd}
\left\|\mathcal{N}\right\|_{L^1_N\cap H^1_N}\leq C\left(\|(v,\psi_x)\|_{H^5_N}^2+\|\psi_t\|_{H^1_N}^2+|\gamma_t|^2\right)\leq C  \zeta(t)^2\left(1+t\right)^{-1/2}
\end{equation}
for some constant $C>0$ independent of $N$.  Using the linear estimates in Proposition \ref{p:lin_est}, as well as the conservative structure of the nonlinearity in \eqref{nonlin_pert}, it follows that
\begin{align*}
\left\|v(t)\right\|_{L^2_N}&\leq \left\|v(t)\psi_x(t)\right\|_{L^2_N} + CE_0\left(1+t\right)^{-1/4}+C\int_0^t(1+t-s)^{-3/4}\left\|\mathcal{N}(s)\right\|_{L^1_N\cap L^2_N}ds\\
&\leq \zeta(t)^2(1+t)^{-1/2} + CE_0(1+t)^{-1/4} + C\zeta(t)^2\int_0^t(1+t-s)^{-3/4}(1+s)^{-1/2}ds\\
&\leq C\left(E_0+\zeta(t)^2\right)(1+t)^{-1/4},
\end{align*}
where, again, the above constant $C>0$ is independent of $N$.  Similarly, from \eqref{e:mod_def} we have for each $1\leq\ell\leq 6$
\[
\partial_x^\ell\psi_x(x,t) = \chi(t)\left[\partial_x^{\ell+1}s_{p,N}(t)v(x,0) + \int_0^t\partial_x^{\ell+1}s_{p,N}(t-s)\partial_x\mathcal{N}(x,s)ds\right]
\]	
and hence
\begin{align*}
\left\|\psi_x\right\|_{H^6_N}&\leq CE_0(1+t)^{-1/4}+C\int_0^t\left(1+t-s\right)^{-3/4}\left\|\partial_x\mathcal{N}(\cdot,s)\right\|_{L^2_N}ds\\
&\leq C\left(E_0+\zeta(t)^2\right)(1+t)^{-1/4}
\end{align*}
again where the constant $C>0$ is independent of $N$.  A completely analogous calculation gives
\[
\left\|\psi_t\right\|_{H^5_N}\leq C\left(E_0+\zeta(t)^2\right)(1+t)^{-1/4},
\]
and an application of integration by parts yields
\[
|\gamma_t(t)|\leq\left|\left<\psi^{adj},\partial_x\mathcal{N}(\cdot,t)\right>_{L^2_N}\right|\leq C\left\|\mathcal{N}(\cdot,t)\right\|_{L^1_N}\leq C\zeta(t)^2(1+t)^{-1/2},
\]
where here we used an $L^1-L^\infty$ bound to control the inner product\footnote{Otherwise, one would get a $\mathcal{O}(N)$-growth from the $\|\partial_x\psi^{adj}\|_{L^2_N}$ term.}.
Using now the nonlinear damping result in Proposition \ref{p:damping}, it follows that
\begin{align*}
\|v(t)\|_{H^5_N}^2&\leq CE_0^2 e^{-\theta t}+C\left(E_0+\zeta(t)^2\right)^2\int_0^t e^{-\theta(t-s)}(1+s)^{-1/2}ds\\
&\leq C\left(E_0+\zeta(t)^2\right)^2 (1+t)^{-1/2}.
\end{align*}
Noting that $\zeta(t)$ is a non-decreasing function, it follows for all $t\in(0,\tau)$ that
\[
\left(\|v(s)\|_{H^5_N}^2+\|\psi_t(s)\|_{H^5_N}^2+\|\psi_x(s)\|_{H^6_N}^2+|\gamma_t(s)|\right)^{1/2}(1+s)^{1/4}\leq C\left(E_0+\zeta(t)\right)^2,
\]
valid for all $s\in(0,t)$.  Taking the supremum over all $s\in(0,t)$ yields the desired result.
\end{proof}

With Proposition \ref{p:induct} established, the proof of Theorem \ref{T:main} follows directly.  Indeed, since $\zeta(t)$ is continuous for so long as it remains
small, it follows from Proposition \ref{p:induct} that if $E_0<\frac{1}{4C}$ then we have $0\leq\zeta(t)\leq 2CE_0$ for all $t\geq 0$.  Noting that the constant $C>0$
from Proposition \ref{p:induct} is independent of $N$ and setting
\[
\widetilde{\psi}(x,t) = \frac{1}{N}\gamma(t)+\psi(x,t),
\]
the stability estimates \eqref{e:stab_estimate} in Theorem \ref{T:main} follow.  Finally, using that $\zeta(t)\leq 2CE_0$ for all $t\geq 0$ it follows
from \eqref{e:Nbd} and applying \eqref{e:lin_bd5} to the implicit representation \eqref{e:mod_def} that
\[
\left\|\psi(\cdot,t)\right\|_{L^\infty}\leq CE_0+CE_0^2\int_0^t(1+t-s)^{-1/2}(1+s)^{-1/2}ds\leq CE_0,
\]
yielding the $L^\infty$ estimate \eqref{e:stab_bd}.

\appendix
\section{Proof of Nonlinear Perturbation Equations}\label{A:pert_eqns}

In this appendix, we present the details of the derivation of the nonlinear perturbation equations in Proposition \ref{p:nonlin_pert}.  To this end,
first set
\[
\widetilde{u}(x,t):={u}\left(x + \left(\delta M\right)t - \frac{1}{N}\gamma(t) - \psi(x,t), t\right)
\]
and note that
\[
\widetilde{u}_x = \frac{1}{1-\psi_x}\d_x(v + \bar{u})
\]
and
\[
\widetilde{u}_t = v_t +  \frac{1}{1-\psi_x}\left(\frac{1}{N}\gamma' + \psi_t - \frac{1}{N}\delta M\right)\d_x(v + \bar{u}).
\]
Since ${u}(x,t)$ is a solution to \eqref{e:ks}, in the traveling coordinate frame $x-ct$, it follows that
\begin{align*}
    &(1-\psi_x)v_t + \left(\frac{1}{N}\gamma' + \psi_t - \frac{1}{N}\delta M\right)\d_x\left(v+\bar{u}\right) - c \d_x\left(v+\bar{u}\right)\\
    &\qquad + \eps\d_x\left(\frac{1}{1-\psi_x}\d_x\left(\frac{1}{1-\psi_x}\d_x\left(v+\bar{u}\right)\right)\right) + \delta\d_x\left(\frac{1}{1-\psi_x}\d_x\left(v + \bar{u}\right)\right)\\
    &\qquad+\delta\d_x\left(\frac{1}{1-\psi_x}\d_x\left(\frac{1}{1-\psi_x}\d_x\left(\frac{1}{1-\psi_x}\d_x\left(v + \bar{u}\right)\right)\right)\right)\\
    &\qquad+\left(v + \bar{u} + \frac{1}{N}\delta M\right)\d_x\left(v+\bar{u}\right) = 0
\end{align*}
Clearly, the contributions from the $\delta M$ terms cancel, which is a reflection of the Galilean invariance of \eqref{e:ks}.
Now, subtracting off the profile equation \eqref{e:profile} for $\bar{u}$ yields
\begin{align*}
    &v_t - (\psi_x v)_t + (\psi_t v)_x + \psi_t\bar{u}' + \left(\frac{1}{N}\gamma'v\right)_x + \frac{1}{N}\gamma'\bar{u}' - cv_x\\
&\quad+ \eps\d_x\left(\frac{1}{1-\psi_x}\d_x\left(\frac{1}{1-\psi_x}\d_x\left(v+\bar{u}\right)\right)\right) - \eps\bar{u}'''
+ \delta\d_x\left(\frac{1}{1-\psi_x}\d_x\left(v + \bar{u}\right)\right)-\delta\bar{u}''\\
    &\qquad+\delta\d_x\left(\frac{1}{1-\psi_x}\d_x\left(\frac{1}{1-\psi_x}\d_x\left(\frac{1}{1-\psi_x}\d_x\left(v + \bar{u}\right)\right)\right)\right)-\delta\bar{u}''''\\
    &\qquad+ \frac{1}{2}\d_x\left(v^2\right) + \d_x \left(\bar{u}v\right) = 0
\end{align*}
or, equivalently\footnote{By adding and subtracting $(\d_t - \mathcal{L})\left(v + \frac{1}{N}\gamma \bar{u}' + \psi\bar{u}'\right)$.}
\begin{align*}
    &(\d_t - \mathcal{L})\left(v + \frac{1}{N}\gamma \bar{u}' + \psi\bar{u}'\right) = (\psi_x v)_t - \left(\psi_t v+ \frac{1}{N}\gamma'v\right)_x- \frac{1}{2}\d_x\left(v^2\right)\\
    &\qquad - \eps\left[\d_x\left(\frac{1}{1-\psi_x}\d_x\left(\frac{1}{1-\psi_x}\bar{u}'\right)\right) - \bar{u}''' -\left(\psi_x\bar{u}'\right)_{xx} - \left(\psi_x\bar{u}''\right)_x\right]\\
    &\qquad - \eps\left[\d_x\left(\frac{1}{1-\psi_x}\d_x\left(\frac{1}{1-\psi_x}v_x\right)\right) - v_{xxx}\right]- \delta\left[\d_x\left(\frac{1}{1-\psi_x}\bar{u}'\right)-\bar{u}'' - \left(\psi_x\bar{u}'\right)_x\right]\\
    &\qquad- \delta\left[\d_x\left(\frac{1}{1-\psi_x}v_x\right)- v_{xx}\right]-\delta\left[\d_x\left(\frac{1}{1-\psi_x}\d_x\left(\frac{1}{1-\psi_x}\d_x\left(\frac{1}{1-\psi_x}v_x\right)\right)\right)- v_{xxxx}\right]\\
    &\qquad-\delta\left[\d_x\left(\frac{1}{1-\psi_x}\d_x\left(\frac{1}{1-\psi_x}\d_x\left(\frac{1}{1-\psi_x}\bar{u}'\right)\right)\right)-\bar{u}'''' - \left(\psi_x\bar{u}'\right)_{xxx} - \left(\psi_x\bar{u}''\right)_{xx} - \left(\psi_x\bar{u}'''\right)_x\right].
\end{align*}
Using the identity $\frac{1}{1-\psi_x} = 1 + \frac{\psi_x}{1-\psi_x}$ and rearranging now yields
\[
(\d_t - \mathcal{L})\left(v + \frac{1}{N}\gamma \bar{u}' + \psi\bar{u}'\right) = \d_x \mathcal{Q} + \d_x\mathcal{R} + (\psi_x v)_t.
\]
Adding and subtracting $\mathcal{L}(\psi_x v)$ completes the proof.

\begin{remark}
One subtle point in our nonlinear analysis is the placement of the modulation functions in the decomposition \eqref{e:pert}.  
From the previous work \cite{SS11,SS11_2} and based on the linear analysis in Section \ref{S:lin}, it may seem more natural to 
use the slightly different decomposition
\begin{equation}\label{e:pert_alt}
\widetilde{u}\left(x-\psi(x,t),t\right)=\bar{u}\left(x-t\left(\delta M\right) + \frac{1}{N}\gamma(t)\right)+\delta M+v\left(x-t\left(\delta M\right),t\right)
\end{equation}
where $\gamma:\RM_+\to\RM$ and $\psi:\RM\times\RM_+\to\RM$ are functions to be determined as in the analysis above. 
However, one can readily check in the derivation of the nonlinear perturbation equations above that the decomposition \eqref{e:pert_alt}
introduces terms that are
\[
\mathcal{O}\left(\frac{1}{N}\gamma(t) + (\delta M)t\right)
\]
and hence, in particular, are not decaying in time.  While this is not a problem in the traditional fixed $N$ theory (with exponential semigroup bounds), 
it is insufficient to close a nonlinear iteration scheme when using the more delicate uniform algebraic bounds.  This is reminiscent of the similar observation
made in \cite[Remark 2.4]{JNRZ_Invent} in the context of localized perturbations of periodic waves in dissipative systems.
\end{remark}

\bibliographystyle{abbrv}
\bibliography{KS}

\end{document}